\documentclass[epsfig,latexsym,amsfonts,twoside]{article}
\usepackage{amssymb,enumerate}
\usepackage{amsmath,amscd}

\def\part#1{\frac{\partial\phantom{#1}}{\partial#1}}
\newtheorem{thm}{Theorem}
\newtheorem{theorem}[thm]{Theorem}
\newtheorem{proposition}[thm]{Proposition}
\newtheorem{lemma}[thm]{Lemma}
\newtheorem{corollary}[thm]{Corollary}
\newtheorem{conjecture}[thm]{Conjecture}
\newtheorem{claim}[thm]{Claim}
\newenvironment{proof}{\begin{trivlist}\item[]{\bf Proof} }%
{\hfill $\Box$ \end{trivlist}}
\newenvironment{definition}{\begin{trivlist}\item[]{\bf Definition}\em }%
{\end{trivlist}}
\newenvironment{remark}{\begin{trivlist}\item[]{\bf Remark} }%
{\end{trivlist}}
\newenvironment{example}{\begin{trivlist}\item[]{\bf Example} }%
{\end{trivlist}}
\newenvironment{question}{\begin{trivlist}\item[]{\bf Question} }%
{\end{trivlist}}

\def\Z{\ifmmode{{\mathbb Z}}\else{${\mathbb Z}$}\fi}
\def\Q{\ifmmode{{\mathbb Q}}\else{${\mathbb Q}$}\fi}
\def\C{\ifmmode{{\mathbb C}}\else{${\mathbb C}$}\fi}
\def\P{\ifmmode{{\mathbb P}}\else{${\mathbb P}$}\fi}

\def\H{\ifmmode{{\mathrm H}}\else{${\mathrm H}$}\fi}

\def\B{\ifmmode{{\mathcal B}}\else{${\mathcal B}$}\fi}
\def\E{\ifmmode{{\mathcal E}}\else{${\mathcal E}$}\fi}
\def\F{\ifmmode{{\mathcal F}}\else{${\mathcal F}$}\fi}
\def\K{\ifmmode{{\mathcal K}}\else{${\mathcal K}$}\fi}
\def\L{\ifmmode{{\mathcal L}}\else{${\mathcal L}$}\fi}
\def\M{\ifmmode{{\mathcal M}}\else{${\mathcal M}$}\fi}
\def\N{\ifmmode{{\mathcal N}}\else{${\mathcal N}$}\fi}
\def\O{\ifmmode{{\mathcal O}}\else{${\mathcal O}$}\fi}
\def\U{\ifmmode{{\mathcal U}}\else{${\mathcal U}$}\fi}
\def\V{\ifmmode{{\mathcal V}}\else{${\mathcal V}$}\fi}
\def\X{\ifmmode{{\mathcal X}}\else{${\mathcal X}$}\fi}

\def\Br{\ifmmode{{\mathrm{Br}}}\else{${\mathrm{Br}}$}\fi}
\def\OG{\ifmmode{\widetilde{\cal M}_4}\else{$\widetilde{\cal M}_4$}\fi}
\def\D{\ifmmode{{\mathcal D}^b}\else{${{\mathcal
    D}^b}$}\fi}
\def\Shah{\ifmmode{\amalg\hspace*{-3.5pt}\amalg}\else{$\amalg\hspace*{-3.5pt}\amalg$}\fi}

\begin{document}

\title{A finiteness theorem for Lagrangian fibrations}
\author{Justin Sawon}
\date{October, 2012}
\maketitle

\begin{abstract}
We consider (holomorphic) Lagrangian fibrations $\pi:X\rightarrow\P^n$ that satisfy some natural hypotheses. We prove that there are only finitely many such Lagrangian fibrations up to deformation.
\end{abstract}

\maketitle

\section{Introduction}

Let $X$ be an irreducible holomorphic symplectic manifold of dimension $2n$, i.e., $X$ is a compact K{\"a}hler manifold admitting a holomorphic two-form $\sigma$ which is non-degenerate in the sense that $\sigma^{\wedge n}$ trivializes the canonical bundle $K_X=\Omega^{2n}$. In this context, irreducibility means that $X$ is simply-connected and $\H^0(X,\Omega^2)$ is generated by $\sigma$. Ultimately we would like to classify such manifolds up to deformation; in this article we assume that $X$ admits the additional structure of a {\em Lagrangian fibration\/}.

Matsushita~\cite{matsushita99} proved that if $\pi:X\rightarrow B$ is a fibration on $X$ then $\mathrm{dim}B$ must equal $n$, the generic fibre must be an $n$-dimensional complex torus, and every fibre must be Lagrangian with respect to the holomorphic symplectic form $\sigma$. Hwang~\cite{hwang08} later proved that the base must be isomorphic to $\P^n$ if it is smooth and $X$ is projective. In this article, by {\em Lagrangian fibration\/} we shall mean an irreducible holomorphic symplectic manifold admitting a fibration $\pi:X\rightarrow\P^n$.

Our main result is:
\begin{theorem}
\label{main}
Fix positive integers $n$ and $d_1,\ldots,d_n$, with $d_1|d_2|\cdots |d_n$. Consider Lagrangian fibrations $\pi:X\rightarrow\P^n$ that satisfy
\begin{enumerate}
\item $\pi:X\rightarrow\P^n$ admits a global section,
\item there is a very ample line bundle on $X$ which gives a polarization of type $(d_1,\ldots,d_n)$ when restricted to a generic smooth fibre $X_t$,
\item over a generic point $t$ of the discriminant locus (i.e., the hypersurface $\Delta\subset\P^n$ parametrizing singular fibres) the fibre $X_t$ is a rank-one semi-stable degeneration of abelian varieties,
\item a neighbourhood $U$ of a generic point $t\in\P^n$ describes a maximal variation of abelian varieties.
\end{enumerate}
Then there are finitely many such Lagrangian fibrations up to deformation.
\end{theorem}
The precise meanings of ``rank-one semi-stable degeneration'' and ``maximal variation'' will be made clear in Section~2, and these hypotheses are discussed further in Subsections~4.1 to~4.3. Some restrictions on the types of polarizations that can occur in dimension $2n=4$ were previously found by the author~\cite{sawon08i}.

The idea of the proof is as follows. The first two hypotheses ensure the existence of a ``classifying map''
$$\phi:\P^n\backslash\Delta\rightarrow\mathcal{A}_{d_1,\ldots,d_n}$$
from the complement of the discriminant locus in the base to the moduli space of $(d_1,\ldots,d_n)$-polarized abelian varieties. The third hypothesis implies that $\phi$ extends over generic points of $\Delta$ to a map
$$\bar{\phi}:\P^n\backslash\Delta_0\rightarrow\mathcal{A}^*_{d_1,\ldots,d_n},$$
where $\Delta_0\subset\Delta$ is codimension two in $\P^n$ and $\mathcal{A}^*_{d_1,\ldots,d_n}$ is the partial compactification of $\mathcal{A}_{d_1,\ldots,d_n}$. We choose an ample line bundle $H$ on $\mathcal{A}^*_{d_1,\ldots,d_n}$, and consider the pull-back $\bar{\phi}^*H$ to $\P^n\backslash\Delta_0$. We establish an upper bound on the degree of $\bar{\phi}^*H$ (the fourth hypothesis ensures that this degree is non-zero), from which we conclude that $\bar{\phi}$ must belong to one of finitely many families of rational maps from $\P^n$ to $\mathcal{A}^*_{d_1,\ldots,d_n}$.

Next we study the space of rational maps from $\P^n$ to $\mathcal{A}^*_{d_1,\ldots,d_n}$. We describe conditions on a map that are necessary and sufficient to ensure that a unique Lagrangian fibration $\pi:X\rightarrow\P^n$ can be reconstructed from it. We prove that each of these conditions is either open, closed, or a finite union of locally closed conditions. In effect, this establishes a bijection between Lagrangian fibrations and a constructible algebraic subset of the space of rational maps from $\P^n$ to $\mathcal{A}^*_{d_1,\ldots,d_n}$ of bounded degree. The latter consists of finitely many connected components, thereby implying the finiteness of Lagrangian fibrations $\pi:X\rightarrow\P^n$ up to deformation.

Theorem~\ref{main} may be compared to a result of Gross~\cite{gross94}, who proved that there are finitely many elliptic Calabi-Yau threefolds up to deformation. The bulk of Gross's article goes into dealing with the relation between a fibration and its relative Jacobian, and in understanding the singular fibres; we avoid both of these difficulties by imposing hypotheses 1 and 3. Our main challenge is to deal with the classifying map, which is relatively simple in the elliptic case. Huybrechts~\cite{huybrechts03} also proved finiteness theorems for holomorphic symplectic manifolds; his results do not require the structure of a Lagrangian fibration, but instead involve $\H^2(X,\mathbb{Z})$ and the Beauville-Bogomolov quadratic form.

The author would like to thank Valery Alexeev, Samuel Grushevsky, Christopher Hacon, J{\'a}nos Koll{\'a}r, and Jun-Muk Hwang for helpful suggestions, and the Korea Institute for Advanced Studies for hospitality. The author gratefully acknowledges support from the NSF, grant number DMS-1206309.

\section{Bounding the degree}

Our goal in this section is to introduce the classifying map of a Lagrangian fibration and establish a bound for its degree.

\subsection{Classifying maps}

We start by recalling some of the theory of moduli spaces of abelian varieties (see Grushevsky~\cite{grushevsky09} and the references contained therein). Consider first the principally polarized case, and let $\mathcal{A}_n:=\mathcal{A}_{1,\ldots, 1}$ be the moduli space of $n$-dimensional principally polarized abelian varieties. Recall that there is a {\em partial compactification\/} of $\mathcal{A}_n$, given as a set by
$$\mathcal{A}^*_n=\mathcal{A}_n\sqcup (\mathcal{X}_{n-1}/\pm 1)$$
where $\mathcal{X}_{n-1}\rightarrow\mathcal{A}_{n-1}$ is the universal principally polarized abelian variety of dimension $n-1$. The codimension one boundary $\mathcal{X}_{n-1}/\pm 1$ parametrizes degenerate abelian varieties in the following way: if $x\in\mathcal{X}_{n-1}$, then $x$ is a point on an $(n-1)$-dimensional abelian variety $A$. Since $A$ is principally polarized, we can also think of $x$ as a point on the dual variety $\hat{A}=\mathrm{Pic}^0A\cong A$, corresponding to a degree zero line bundle $\mathcal{L}_x$ on $A$. Take the $\P^1$-bundle $\P(\O_A\oplus\mathcal{L}_x)$ over $A$ and glue the zero and infinity sections; these sections are both isomorphic to $A$, under the projection $\P(\O_A\oplus\mathcal{L}_x)\rightarrow A$, but when we glue them we include a translation by $x\in A$. Let $Y$ denote the resulting variety. Note that $Y$ is not normal: topologically it looks like the product of $A$ with a nodal rational curve, but as a complex algebraic variety $Y$ is only a product if $x=0$. Starting instead with $-x$ produces the same variety, so the construction only depends on $[x]\in\mathcal{X}_{n-1}/\pm 1$.
\begin{definition}
The resulting variety $Y$ is known as a rank-one semi-stable degeneration of principally polarized abelian varieties.
\end{definition}
Note that the smooth locus of $Y$, which is a $\C^*$-bundle over $A$, is a semi-abelian variety. Then $Y$ is a toric compactification of this semi-abelian variety, with rank-one toric part.

For a general polarization of type $(d_1,\ldots,d_n)$, a rank-one semi-stable degeneration can consist of several irreducible components. Each component will be a $\P^1$-bundle over an $(n-1)$-dimensional abelian variety $A$ (unless there is only one component), and they will be glued together to form a cycle, with the infinity section of one glued to the zero section of the next. Overall, the composition of all the gluings will include a translation by an element $x\in A$. There is a partial compactification $\mathcal{A}^*_{d_1,\ldots,d_n}$ of the moduli space $\mathcal{A}_{d_1,\ldots,d_n}$ of $(d_1,\ldots,d_n)$-polarized abelian varieties, whose boundary divisor
$$\mathcal{A}^*_{d_1,\ldots,d_n}\backslash \mathcal{A}_{d_1,\ldots,d_n}$$
parametrizes these rank-one semi-stable degenerations. To determine all the possibilities for the number of components of a degeneration and the polarization type of $A$ is a complicated combinatorial problem, but we do not need an explicit solution here. It is enough to note that each possibility yields an irreducible component of the boundary divisor.

Now suppose that $Z\rightarrow U$ is a family of $(d_1,\ldots,d_n)$-polarized abelian varieties over a small open set $U$, in the analytic topology. There is a map $U\rightarrow\mathcal{A}_{d_1,\ldots,d_n}$ associated to this family.
\begin{definition}
We say that $Z/U$ describes a maximal variation of abelian varieties if the map $U\rightarrow\mathcal{A}_{d_1,\ldots,d_n}$ is an immersion.
\end{definition}

Having made these clarifications, let us turn to the situation of Theorem~\ref{main}, i.e., suppose that $\pi:X\rightarrow\P^n$ is a Lagrangian fibration satisfying hypotheses~1 to~4. Over the complement $\P^n\backslash\Delta$ of the discriminant locus we have a family of $n$-dimensional complex tori. The first hypothesis, existence of a global section, rigidifies this family by providing a basepoint in each fibre. By the second hypothesis, the fibres are $(d_1,\ldots,d_n)-$polarized. We therefore have a map
$$\phi:\P^n\backslash\Delta\rightarrow\mathcal{A}_{d_1,\ldots,d_n}.$$
By the third hypothesis, the singular fibre of $X\rightarrow\P^n$ over a generic point of $\Delta$ is a rank-one semi-stable degeneration, and therefore $\phi$ extends over generic points of $\Delta$ to a map
$$\bar{\phi}:\P^n\backslash\Delta_0\rightarrow\mathcal{A}^*_{d_1,\ldots,d_n},$$
where $\Delta_0\subset\Delta$ is codimension at least two in $\P^n$.
Finally, $\phi$ is an immersion by the fourth hypothesis. Note, however, that $\bar{\phi}$ might not be an immersion along $\Delta\backslash\Delta_0$.
\begin{definition}
We call $\bar{\phi}$ the classifying map of the Lagrangian fibration $\pi:X\rightarrow\P^n$.
\end{definition}


\subsection{Line bundles}

We first consider line bundles on the moduli space $\mathcal{A}_n$ of principally polarized abelian varieties and its partial compactification $\mathcal{A}^*_n$. The {\em Hodge line bundle\/} on $\mathcal{A}_n$ is defined as
$$L:=\pi_*\omega_{\mathcal{X}_n/\mathcal{A}_n}$$
where $\pi:\mathcal{X}_n\rightarrow\mathcal{A}_n$ is the universal abelian variety and $\omega_{\mathcal{X}_n/\mathcal{A}_n}$ is the relative canonical bundle. The local system $R^n\pi_*\mathbb{C}_{\mathcal{X}_n/\mathcal{A}_n}$ is a variation of Hodge structure over $\mathcal{A}_n$, and $L$ is the bottom piece of the Hodge filtration. In other words, we have an embedding as a sub-bundle
$$L\hookrightarrow \mathcal{O}_{\mathcal{A}_n}\otimes_{\mathbb{C}}R^n\pi_*\mathbb{C}_{\mathcal{X}_n/\mathcal{A}_n}.$$
The Hodge line bundle can be extended to $\mathcal{A}^*_n$ in a natural way. Namely, $\mathcal{O}_{\mathcal{A}_n}\otimes_{\mathbb{C}}R^n\pi_*\mathbb{C}_{\mathcal{X}_n/\mathcal{A}_n}$ has an upper canonical extension to a locally free sheaf on $\mathcal{A}^*_n$, and $\pi_*\omega_{\mathcal{X}^*_n/\mathcal{A}^*_n}$ coincides with the corresponding upper canonical extension of the bottom piece of the Hodge filtration (see Theorem~8.10.7 of Koll{\'a}r~\cite{kollar07}),
where $\pi:\mathcal{X}^*_n\rightarrow\mathcal{A}^*_n$ is the extension of the universal abelian variety to the partial compactification $\mathcal{A}^*_n$ (we are abusing notation, using $\pi$ to denote both projections $\mathcal{X}_n\rightarrow\mathcal{A}_n$ and $\mathcal{X}^*_n\rightarrow\mathcal{A}^*_n$, but the meaning should be unambiguous). Then $\pi_*\omega_{\mathcal{X}^*_n/\mathcal{A}^*_n}$, which we will also denote by $L$, is the required extension of the Hodge line bundle to $\mathcal{A}^*_n$.

\begin{remark}
The Hodge line bundle $L$ on $\mathcal{A}_n$ can also be defined as the bundle of scalar modular forms of weight one, which again admits a natural extension to $\mathcal{A}^*_n$ (see Grushevsky~\cite{grushevsky09}).
\end{remark}

The boundary divisor
$$D:=\mathcal{X}_{n-1}/\pm 1=\mathcal{A}^*_n\backslash\mathcal{A}_n$$
in $\mathcal{A}^*_n$ is irreducible. The Picard group of $\mathcal{A}^*_n$ is generated over $\mathbb{Q}$ by $L$ and $D$ (we will use the same notation to denote both divisors and their associated line bundles; the intended meaning will be clear from context). Note that since $\mathcal{A}_n$ and $\mathcal{A}^*_n$ are really moduli stacks, we work with $\mathbb{Q}$-divisors. The nef and ample cones of $\mathcal{A}^*_n$ are determined by the following result.
\begin{theorem}[Hulek-Sankaran~\cite{hs04}]
The $\mathbb{Q}$-divisor $aL-bD$ on $\mathcal{A}^*_n$ is nef iff $a\geq 12b\geq 0$ and ample iff $a>12b>0$.
\end{theorem}
It follows that $H=aL-bD$ will be very ample on $\mathcal{A}^*_n$ for sufficiently large $a$ and $b$ satisfying $a>12b>0$.

For general polarizations, let $\pi :\mathcal{X}\rightarrow\mathcal{A}_{d_1,\ldots,d_n}$ be the universal abelian variety and let $\pi :\mathcal{X}^*\rightarrow\mathcal{A}^*_{d_1,\ldots,d_n}$ be its extension to the partial compactification. The Hodge line bundle can be defined on $\mathcal{A}_{d_1,\ldots,d_n}$ in the same way as before, i.e.,
$$L:=\pi_*\omega_{\mathcal{X}/\mathcal{A}_{d_1,\ldots,d_n}},$$
and extended to the partial compactification by
$$L:=\pi_*\omega_{\mathcal{X}^*/\mathcal{A}^*_{d_1,\ldots,d_n}}.$$
The boundary
$$D:=\mathcal{A}^*_{d_1,\ldots,d_n}\backslash\mathcal{A}_{d_1,\ldots,d_n}$$
will no longer be irreducible, so we write
$$D=\bigcup_{i=1}^mD_i$$
for the decomposition into irreducible components. Then $L$ and $D_1,\ldots,D_m$ generate the Picard group of $\mathcal{A}^*_{d_1,\ldots,d_n}$ over $\mathbb{Q}$. Moreover,
$$H=aL-b_1D_1-\ldots -b_mD_m$$
will be very ample for sufficiently large $a,b_1,\ldots,b_m$, with $a\gg b_i$. We will fix such a divisor $H$ for the remainder of the article.



We return now to the Lagrangian fibration $\pi:X\rightarrow\P^n$. Recall that we have a classifying map
$$\bar{\phi}:\P^n\backslash\Delta_0\rightarrow\mathcal{A}^*_{d_1,\ldots,d_n}.$$
We wish to describe the pullbacks of $L$ and $D$ by $\bar{\phi}$.

\begin{lemma}
The pullback to $\P^n\backslash\Delta_0$ of the Hodge line bundle is
$$\bar{\phi}^*L\cong\mathcal{O}(n+1)|_{\P^n\backslash\Delta_0}.$$
\end{lemma}

\begin{proof}
Because the boundary divisor $D$ of $\mathcal{A}^*_{d_1,\ldots,d_n}$ parametrizes rank-one semi-stable degenerations of abelian varieties, the local system $R^n\pi_*\mathbb{C}_{\mathcal{X}/\mathcal{A}_{d_1,\ldots,d_n}}$ will have unipotent monodromy around every irreducible component of $D$. Similarly, hypothesis $3$ implies that the corresponding local system on $\P^n\backslash\Delta$ will have unipotent monodromy around every irreducible component of $\Delta$. By Theorem~8.10.8 of Koll{\'a}r~\cite{kollar07} the canonical extension of the bottom piece of the Hodge filtration commutes with pullbacks in this situation. In other words
$$\bar{\phi}^*L=\bar{\phi}^*\pi_*\omega_{\mathcal{X}^*/\mathcal{A}^*_{d_1,\ldots,d_n}}\cong\pi_*\omega_{X/(\P^n\backslash\Delta_0)}$$
(we are abusing notation again, by using $\pi$ to denote both the projections $X\rightarrow\P^n\backslash\Delta_0$ and $\mathcal{X}^*\rightarrow\mathcal{A}^*_{d_1,\ldots,d_n}$).

If $\pi:X\rightarrow\P^n$ were smooth everywhere, we would have by Grothendieck-Serre duality
\begin{eqnarray*}
R\pi_*\omega_{X/\P^n}[n] & \cong & R\pi_*\mathcal{H}om_X(\mathcal{O}_X,\omega_{X/\P^n}[n]) \\
& \cong & R\pi_*\mathcal{H}om_X(\mathcal{O}_X,\pi^!\mathcal{O}_{\P^n}) \\
& \cong & R\mathcal{H}om_{\P^n}(R\pi_*\mathcal{O}_X,\mathcal{O}_{\P^n}) \\
& \cong & (R\pi_*\mathcal{O}_X)^{\vee},
\end{eqnarray*}
and in particular $\pi_*\omega_{X/\P^n}\cong(R^n\pi_*\mathcal{O}_X)^{\vee}$. Of course $\pi:X\rightarrow\P^n$ is not smooth everywhere, but this isomorphism still holds over $\P^n\backslash\Delta$, i.e., away from the singular fibres, a priori. Moreover, since $\pi_*\omega_{X/(\P^n\backslash\Delta)}$ and $R^n\pi_*\mathcal{O}_X|_{\P^n\backslash\Delta}$ are the bottom and top pieces of the Hodge filtration on the local system $R^n\pi_*\mathbb{C}_{X/(\P^n\backslash\Delta)}$, they admit canonical extensions from $\P^n\backslash\Delta$ to $\P^n\backslash\Delta_0$, and the isomorphism extends to
$$\pi_*\omega_{X/(\P^n\backslash\Delta_0)}\cong(R^n\pi_*\mathcal{O}_X)^{\vee}|_{\P^n\backslash\Delta_0}.$$

Finally, Proposition~7.6 of Koll{\'a}r~\cite{kollar86i} states that $R^n\pi_*\omega_X\cong\omega_Y$ for a fibration $\pi:X\rightarrow Y$ of relative dimension $n$ between smooth projective varieties. In our case $X$ is a holomorphic symplectic manifold, so $\omega_X\cong\mathcal{O}_X$ and Koll{\'a}r's result becomes $R^n\pi_*\mathcal{O}_X\cong\omega_{\P^n}$ (in fact, Matsushita~\cite{matsushita05} proved that $R^i\pi_*\mathcal{O}_X\cong\Omega^i_{\P^n}$ for all $i$, though we do not need this stronger statement). Putting everything together we obtain
$$\bar{\phi}^*L\cong\pi_*\omega_{X/(\P^n\backslash\Delta_0)}\cong(R^n\pi_*\mathcal{O}_X)^{\vee}|_{\P^n\backslash\Delta_0}\cong(\omega_{\P^n})^{\vee}|_{\P^n\backslash\Delta_0}\cong\mathcal{O}(n+1)|_{\P^n\backslash\Delta_0}.$$
\end{proof}

\begin{remark}
There is an alternate approach to calculating $\bar{\phi}^*L$, which is to use the generalization of Kodaira's canonical bundle formula as proved by Kawamata~\cite{kawamata98}, and described by Koll{\'a}r in~\cite{kollar07}. This formula applies to a fibration $f:X\rightarrow Y$ between normal varieties with generic fibre $F$ which satisfies $p_g^+=1$, and it looks like
\begin{eqnarray}
\label{Kodaira}
K_X+R & \sim_{\mathbb{Q}} & f^*(K_Y+J(X/Y,R)+B_R).
\end{eqnarray}
The full statement is rather complicated, and takes about half a page to state (see Theorem~8.5.1 of~\cite{kollar07}), but the main points are
\begin{itemize}
\item the $\mathbb{Q}$-divisor $R$ is chosen so that $K_X+R$ is $\mathbb{Q}$-Cartier and rationally equivalent to the pullback of some $\mathbb{Q}$-Cartier divisor on $Y$,
\item the reduced divisor $B$ on $Y$ parametrizes ``bad'' singular fibres, in the sense that $f$ has slc (semilog canonical) fibres in codimension one over $Y\backslash B$,
\item these bad singular fibres then contribute to the formula through the $\mathbb{Q}$-divisor $B_R$, supported on $B$,
\item the term $J(X/Y,R)$ is the moduli part, describing the variation of the fibres in the family $f:X\rightarrow Y$.
\end{itemize}

Now for a Lagrangian fibration $\pi:X\rightarrow\P^n$, $K_X$ is trivial, so $R$ can be chosen to be trivial too. By our hypothesis $3$, the codimension one singular fibres of $\pi$ are rank-one semi-stable degenerations of abelian varieties, which are slc. So $B$ is empty, and then $B_R$ must also vanish. The canonical bundle formula therefore looks like
$$0=K_X\sim_{\mathbb{Q}}\pi^*(K_{\P^n}+J(X/\P^n)),$$
and we conclude that the moduli part $J(X/\P^n)\cong K_{\P^n}^{\vee}$. But this moduli part can be identified with the pullback of the Hodge line bundle (see Definition~8.4.6 and the base-change Proposition~8.4.9 in~\cite{kollar07}).
\end{remark}

Over each irreducible component of $\Delta$ the fibres of $X\rightarrow\P^n$ will conform to a particular combinatorial type of degeneration, corresponding to a particular irreducible component of the boundary divisor $D$ in $\mathcal{A}^*_{d_1,\ldots,d_n}$. Thus we can write
$$\Delta=\bigcup_{i=1}^m\bigcup_{j=1}^{s_i}\Delta_{ij}$$
for the decomposition of $\Delta$ into irreducible components, where $\bar{\phi}$ maps $\Delta_{ij}$ into $D_i$. Note that $s_i$ could be zero for some $i$, if $X\rightarrow\P^n$ contains no singular fibres of the type parametrized by $D_i$.

\begin{lemma}
The pullback to $\P^n\backslash\Delta_0$ of the irreducible component $D_i$ of the boundary divisor is
$$\bar{\phi}^*D_i=(k_{i1}\Delta_{i1}+\ldots+k_{is_i}\Delta_{is_i})|_{\P^n\backslash\Delta_0}$$
for some integers $k_{ij}\geq 1$.
\end{lemma}

\begin{proof}
This follows immediately from the fact that $D_i$ and the $\Delta_{ij}$ both parametrize degenerate abelian varieties of the same combinatorial type. Note that $k_{ij}$ could be strictly greater than $1$ if $\bar{\phi}$ is ramified along $\Delta_{ij}$.
\end{proof}

\begin{remark}
In fact, there should not be ramification along any $\Delta_{ij}$, as ramification would imply that a local base change is possible, leading to stable reduction of the singular fibres over $\Delta_{ij}$ (see Hwang and Oguiso~\cite{ho11}). However, the singular fibres over $\Delta_{ij}$ are already semi-stable by hypothesis $3$, and therefore no stable reduction should be needed.
\end{remark}


Combining the previous two lemmas we reach the following conclusion.

\begin{lemma}
\label{bound}
The pullback $\bar{\phi}^*H$ of the very ample bundle $H$ on $\mathcal{A}^*_{d_1,\ldots,d_n}$ has bounded degree, where we measure the degree by intersecting $\bar{\phi}^*H$ with a generic line in $\P^n$ (which will of course avoid the codimension two subset $\Delta_0$).
\end{lemma}

\begin{proof}
Remember that we have fixed sufficiently large integers $a$ and $b_1,\ldots,b_m$ so that $H=aL-b_1D_1-\ldots -b_mD_m$ is very ample. Then
\begin{eqnarray*}
\bar{\phi}^*H & = & a\bar{\phi}^*L-\sum_{i=1}^m b_i\bar{\phi}^*D_i \\
 & \cong & \left.\mathcal{O}\left(a(n+1)-\sum_{i=1}^m\sum_{j=1}^{s_i}b_ik_{ij}\mathrm{deg}\Delta_{ij}\right)\right|_{\P^n\backslash\Delta_0}
 \end{eqnarray*}
has degree bounded above by $a(n+1)$.
\end{proof}

\begin{corollary}
For a Lagrangian fibration by principally polarized abelian varieties the degree of the discriminant locus is bounded
$$\mathrm{deg}\Delta\leq 12(n+1).$$
\end{corollary}

\begin{proof}
Since $H$ is very ample and $\bar{\phi}$ is generically an immersion, we must have
$$\mathrm{deg}(\bar{\phi}^*H)=a(n+1)-\sum_{i=1}^m\sum_{j=1}^{s_i}b_ik_{ij}\mathrm{deg}\Delta_{ij}>0.$$
In the principally polarized case the boundary divisor $D$ is irreducible, so $m=1$, and dropping $i$ from the notation we have
$$a(n+1)-b\sum_{j=1}^sk_j\mathrm{deg}\Delta_j>0.$$
Therefore
$$\mathrm{deg}\Delta=\sum_{j=1}^s\mathrm{deg}\Delta_j\leq\sum_{j=1}^sk_j\mathrm{deg}\Delta_j<\frac{a(n+1)}{b}.$$
Since this is true for all sufficiently large $a$ and $b$ satisfying $a>12b>0$, the lemma follows.
\end{proof}

\begin{remark}
Elliptic K3 surfaces can have up to 24 singular fibres, so this bound is sharp when $n=1$. However, it is not sharp when $n=2$, as the author proved in~\cite{sawon08i} that $\mathrm{deg}\Delta\leq 30$ for Lagrangian fibrations by abelian surfaces (assuming the generic singular fibres are semi-stable, but allowing arbitrary polarizations of the fibres). In higher dimensions, $\mathrm{deg}\Delta=6(n+3)$ for the Beauville-Mukai integrable system, and $\mathrm{deg}\Delta=6(n+1)$ for the Lagrangian fibration on the generalized Kummer variety (whose fibres are not principally polarized).
\end{remark}

\section{Finiteness}

To each Lagrangian fibration we have associated a classifying map. In this section we show that this association is one-to-one, i.e., two Lagrangian fibrations with the same classifying map must be isomorphic. Then we give a precise description of the space of all classifying maps. This will ultimately lead to a proof of Theorem~\ref{main}.

\subsection{Injectivity}

\begin{proposition}
\label{injective}
The association that assigns the classifying map 
$$\bar{\phi}:\P^n\backslash\Delta_0\rightarrow\mathcal{A}^*_{d_1,\ldots,d_n}$$
to each Lagrangian fibration $\pi:X\rightarrow\P^n$ satisfying hypotheses 1 to 4 of Theorem~\ref{main} is one-to-one.
\end{proposition}

\begin{remark}
Strictly speaking, we should consider Lagrangian fibrations up to isomorphism and classifying maps up to the natural action of $\mathrm{PGL}(n+1,\mathbb{C})$ on the space of such maps. We will leave these identifications implicit.
\end{remark}

\begin{proof}
For the sake of analogy, we begin with a related but simpler problem: reconstructing an elliptic fibration $Y\rightarrow B$ from its functional invariant or $j$-function. The following theory was developed by Kodaira; see Section~V.11.\ of Barth, Peters, and Van de Ven~\cite{bpv84}. As with our Lagrangian fibrations, we impose several hypotheses on the elliptic fibration: namely, it admits a global section and the generic singular fibres are stable (i.e., of type $I_b$ in Kodaira's classification). Let us again denote the discriminant locus by $\Delta\subset B$ and the codimension two subset of non-generic singular fibres by $\Delta_0\subset B$. Associated to an elliptic fibration is its functional invariant
$$J:B\backslash\Delta_0\rightarrow\C\cup\{\infty\}\cong\overline{\mathcal{M}}_{1,1}$$
where $\overline{\mathcal{M}}_{1,1}$ is the compactification of the moduli space of elliptic curves, i.e., genus one curves with one marked point, and its homological invariant, a representation
$$R:\pi_1(B\backslash\Delta)\rightarrow\mathrm{SL}(2,\mathbb{Z})$$
which describes the monodromy as we pass around $\Delta$. Now composing $R$ with the quotient
$$\mathrm{SL}(2,\mathbb{Z})\rightarrow\mathrm{SL}(2,\mathbb{Z})/\{\pm 1\}=\mathrm{PGL}(2,\mathbb{Z})$$
gives a representation
$$r:\pi_1(B\backslash\Delta)\rightarrow\mathrm{PGL}(2,\mathbb{Z})$$
which is already uniquely determined by the functional invariant. Thus a homological invariant belonging to $J$ is really just a lift of $r$ to $R$.

First consider elliptic surfaces, so that $\mathrm{dim}B=1$, $\Delta$ is a collection of points, and $\Delta_0$ is empty. In general, an elliptic surface $Y\rightarrow B$ which admits a global section (in particular, without multiple fibres) and with possibly non-stable singular fibres is uniquely determined by $J$ and $R$ (see Theorem~11.1 of~\cite{bpv84}). Now up to conjugation, the monodromies around fibres of type $I_b$ and $I^*_b$ are
$$\left(\begin{array}{cc} 1 & b \\ 0 & 1 \end{array}\right)\qquad\mbox{and}\qquad -\left(\begin{array}{cc} 1 & b \\ 0 & 1 \end{array}\right)$$
respectively. So if there is a lift of $r$ to $R$ such that all singular fibres of the corresponding elliptic surface are stable, i.e., of type $I_b$, then any other lift of $r$ to $R$ would lead to singular fibres of type $I^*_b$, i.e., unstable. In other words, by imposing the hypothesis that the singular fibres are stable, we ensure that the elliptic surface is uniquely determined by its functional invariant $J$ alone.

The same arguments applied to higher dimensional elliptic fibrations, with $\mathrm{dim}B>1$, enable us to uniquely determine the fibration over $B\backslash\Delta_0$ from $J$ and $R$ for fibrations admitting global sections, and from $J$ alone for fibrations whose generic singular fibres are stable. To conclude, we need to extend over the codimension two subset $\Delta_0$. This follows from a result of Matsusaka and Mumford~\cite{mm64} which states that a birational map between smooth projective varieties which restricts to an isomorphism on the complements of codimension two subsets, and which takes an ample divisor to an ample divisor, must in fact be an isomorphism (see Exercise~5.6 of Koll{\'a}r, Smith, and Corti~\cite{ksc04}). There is also a relative version of this result, with ample divisors replaced by relatively ample divisors (see Exercise~75 of Koll{\'a}r~\cite{kollar10}). In our case, given two elliptic fibrations $Y\rightarrow B$ and $Y^{\prime}\rightarrow B$ with the same functional invariant (and which admit global sections and whose generic singular fibres are stable) we first conclude that they are isomorphic over $B\backslash\Delta_0$ and then apply the relative version of Matsusaka and Mumford's result to conclude $Y\cong Y^{\prime}$.

Returning to Lagrangian fibrations, first note that the classifying map $\bar{\phi}$ takes the place of the functional invariant $J$. In the principally polarized case, the classifying map will determine a representation
$$r:\pi_1(\P^n\backslash\Delta)\rightarrow\mathrm{PSp}(2n,\mathbb{Z})$$
and the homological invariant will be a lift
$$R:\pi_1(\P^n\backslash\Delta)\rightarrow\mathrm{Sp}(2n,\mathbb{Z})$$
representing the monodromy around $\Delta$. For general polarizations, $\mathrm{PSp}(2n,\mathbb{Z})$ and $\mathrm{Sp}(2n,\mathbb{Z})$ should be replaced by the appropriate discrete symplectic groups of type $(d_1,\ldots,d_n)$. Suppose the monodromy around a semi-stable singular fibre is given by the matrix $T$. If we choose a different lift of $r$ to $R$, the monodromy will change to $\zeta T$, where $\zeta$ is some $m^{\mathrm{th}}$ root of unity. By the stable reduction theory developed by Hwang and Oguiso~\cite{ho11}, the monodromy $\zeta T$ will correspond to an unstable singular fibre, as a base change of order $m$ will be necessary to produce a fibration with a semi-stable singular fibre (note that their terminology ``stable'' is equivalent to our ``semi-stable''). We conclude that, as with elliptic fibrations, if there is a lift of $r$ to $R$ such that all generic singular fibres of the corresponding Lagrangian fibration are semi-stable, then any other lift of $r$ to $R$ would lead to unstable singular fibres. Therefore a Lagrangian fibration whose generic singular fibres are semi-stable is uniquely determined over $\P^n\backslash\Delta_0$ by its classifying map $\bar{\phi}$ alone. Note that in the principally polarized case, Hwang and Oguiso~\cite{ho10} actually give an explicit construction of the Lagrangian fibration from $\bar{\phi}$, which they call the ``period map'', but for our argument their stable reduction theory~\cite{ho11}, which applies to all polarization types, is sufficient.

Finally, we can extend our result over $\Delta_0$ by applying Matsusaka and Mumford's result as before. This concludes the proof.
\end{proof}

\begin{remark}
Note that for an arbitrary functional invariant $J:B\backslash\Delta_0\rightarrow\C\cup\{\infty\}$ there may be no lift of $r$ to $R$ whose corresponding elliptic fibration has generic fibres that are all stable. Moreover, if such a lift does exist, it may not be possible to complete the elliptic fibration over $B\backslash\Delta_0$ to a fibration over all of $B$. Similar comments apply to the classifying map $\bar{\phi}$. In the next subsection, we will identify conditions on $\bar{\phi}$ that ensure that a complete Lagrangian fibration $X\rightarrow\P^n$ whose generic fibres are semi-stable can be recovered from $\bar{\phi}$.
\end{remark}

\subsection{The space of classifying maps}

Recall that we have fixed $a$ and $b_1,\ldots,b_m$ such that $H=aL-b_1D_1-\ldots-b_mD_m$ is a very ample divisor on $\mathcal{A}^*_{d_1,\ldots,d_n}$. We therefore have an embedding
$$\mathcal{A}^*_{d_1,\ldots,d_n}\hookrightarrow\P^N$$
such that $H$ is a hyperplane section of $\mathcal{A}^*_{d_1,\ldots,d_n}$. The composition of $\bar{\phi}$ with this embedding is a map
$$\psi:\P^n\backslash\Delta_0\rightarrow\P^N,$$
which is given by $N+1$ homogeneous polynomials $\psi_i(z_0,\ldots,z_n)$, all of the same degree
$$k:=\mathrm{deg}(\psi^*\mathcal{O}_{\P^N}(1))=\mathrm{deg}(\bar{\phi}^*H).$$
In other words,
$$\psi=(\psi_0,\ldots,\psi_N)\in V_k:=\mathbb{C}[z_0,\ldots,z_n]_k^{\oplus (N+1)}.$$
In this section, by {\em classifying map\/} we shall mean $\psi$ rather than $\bar{\phi}$.

By Lemma~\ref{bound}, $k$ is bounded by $a(n+1)$, so the space of all classifying maps lies in $\bigsqcup_{k=1}^{a(n+1)}V_k$. However, not every map in $\bigsqcup_{k=1}^{a(n+1)}V_k$ will arise as a classifying map, so we need to describe the subset of classifying maps. It is enough to focus on a particular degree; henceforth we assume that $k$ is fixed and we drop it from our notation, writing simply $V$ for $V_k$. We will describe a sequence of conditions on a map $\psi\in V$ that are both necessary and sufficient for $\psi$ to be the classifying map of a Lagrangian fibration. At each step, we will prove that the added condition is either open, closed, or locally closed (more precisely, in this last case the subset of maps satisfying the condition will be a finite union of locally closed subsets). Ultimately, we will find that the subset $V^{(7)}\subset V$ of classifying maps is a constructible algebraic set, and therefore it consists of finitely many connected components.

Firstly, as a rational map from $\P^n$ to $\P^N$, the classifying map should be defined up to codimension two.

\begin{definition}
Define $V^{(1)}\subset V$ to be the subset of maps
$$\psi=(\psi_0,\ldots,\psi_N)\in V=\mathbb{C}[z_0,\ldots,z_n]_k^{\oplus (N+1)}$$
such that $\psi_i$ do not simultaneously vanish on any codimension one subvariety of $\P^n$, and which therefore define a rational map $\psi:\P^n\dashrightarrow\P^N$ that is regular on the complement of a codimension two subset $\Delta_0\subset\P^n$. (Note that the codimension two subset $\Delta_0$ is not fixed, but depends on $\psi$; see the remark following Lemma~\ref{3in2}.)
\end{definition}

\begin{lemma}
The subset $V^{(1)}$ is (Zariski) open in $V$.
\end{lemma}

\begin{proof}
This follows from a straightforward determinantal argument.
\end{proof}

By hypothesis $4$, we are considering only Lagrangian fibrations which are maximal variations (around generic fibres). Therefore the classifying map should be an immersion.

\begin{definition}
Define $V^{(2)}\subset V^{(1)}$ to be the subset of maps such that $\psi:\P^n\dashrightarrow\P^N$ is an immersion at a generic point.
\end{definition}

\begin{lemma}
The subset $V^{(2)}$ is (Zariski) open in $V^{(1)}$.
\end{lemma}

\begin{proof}
This also follows from a straightforward determinantal argument.
\end{proof}

Next, the classifying map must factor through the embedding of the partial compactification of the moduli space of abelian varieties in $\P^N$.

\begin{definition}
Define $V^{(3)}\subset V^{(2)}$ to be the subset of maps such that the image of $\psi$ lies in $\mathcal{A}^*_{d_1,\ldots,d_n}\subset\P^N$ but {\em not\/} entirely inside the boundary divisor $D\subset\mathcal{A}^*_{d_1,\ldots,d_n}$.
\end{definition}

\begin{lemma}
\label{3in2}
The subset $V^{(3)}$ is (Zariski) locally closed in $V^{(2)}$, i.e., open in its closure.
\end{lemma}

\begin{proof}
Consider first the subset of maps in $V^{(2)}$ such that the image of $\psi$ lies in the closure of $\mathcal{A}^*_{d_1,\ldots,d_n}$ in $\P^N$. This is clearly a closed condition, and in fact this subset is the closure $\bar{V}^{(3)}$ of $V^{(3)}$. Since $\mathcal{A}^*_{d_1,\ldots,d_n}$ is open in its closure, and since the complement $\mathcal{A}_{d_1,\ldots,d_n}$ of the boundary divisor $D$ is open in $\mathcal{A}^*_{d_1,\ldots,d_n}$, $V^{(3)}$ will be open in $\bar{V}^{(3)}$.
\end{proof}

\begin{remark}
Each $\psi\in V^{(3)}$ gives a rational map $\psi:\P^n\dashrightarrow\mathcal{A}^*_{d_1,\ldots,d_n}$, regular on $\P^n\backslash\Delta_0$, whose image is not contained in the boundary divisor $D\subset\mathcal{A}^*_{d_1,\ldots,d_n}$. Thus the closure of $\psi^{-1}(D)$ is a divisor in $\P^n$, which we denote by $\Delta$.

Again, $\Delta\subset\P^n$ is not fixed, but depends on $\psi$. Indeed, we can pull back the boundary divisor $D$ by the universal map
$$V^{(3)}\times\P^n\dashrightarrow\mathcal{A}^*_{d_1,\ldots,d_n}$$
to obtain a universal discriminant
$$\mathcal{D}\subset V^{(3)}\times\P^n$$
which gives the appropriate $\Delta$ when restricted to each $\{\psi\}\times\P^n$. The indeterminacy of the universal map also yields
$$\mathcal{D}_0\subset\mathcal{D}\subset V^{(3)}\times\P^n,$$
giving the appropriate $\Delta_0$ for each $\psi\in V^{(3)}$. 
\end{remark}

If $\psi\in V^{(3)}$ then we can pull back the universal family $\mathcal{X}^*\rightarrow\mathcal{A}^*_{d_1,\ldots,d_n}$ by $\psi$, to obtain a family $X\rightarrow\P^n\backslash\Delta_0$ of abelian varieties of polarization type $(d_1,\ldots,d_n)$, and rank-one semi-stable degenerations thereof. Note that $\mathcal{A}_{d_1,\ldots,d_n}$ is not a fine moduli space, so there could be more than one family of abelian varieties $X\rightarrow\P^n\backslash\Delta$ associated to the map $\psi|_{\P^n\backslash\Delta}$ (where, by the definition of $\Delta$, $\P^n\backslash\Delta$ is the inverse image of $\mathcal{A}_{d_1,\ldots,d_n}$ under $\psi$). As described in the previous section, these different families would have the same classifying map but different homological invariants. However, we saw that the homological invariant is uniquely determined by the classifying map if the generic singular fibres are required to be semi-stable. Therefore the above family $X\rightarrow\P^n\backslash\Delta_0$ is uniquely determined.

Next we need to add a condition that will ensure that the family $X\rightarrow\P^n\backslash\Delta_0$ can be completed to a family over $\P^n$. We have a (regular) map from $\P^n\backslash\Delta_0$ to $\mathcal{A}^*_{d_1,\ldots,d_n}$, but it is too much to expect this to extend to a map from $\P^n$ to, say, a toroidal compactification of $\mathcal{A}_{d_1,\ldots,d_n}$ (indeed, the author knows no examples in dimension $2n\geq 4$ where this happens; Lagrangian fibrations always seem to contain unstable fibres in higher codimension). On the other hand, the fibres $X_t$ are polarized and can be embedded in projective space. So instead of using compactifications of moduli spaces of abelian varieties, we will use Hilbert schemes.

A Lagrangian fibration $\pi:X\rightarrow\P^n$ is equidimensional, and therefore a flat family. Our hypothesis $2$ ensures the existence of a very ample line bundle $A$ on $X$ which restricts to a polarization $A_t$ of type $(d_1,\ldots,d_n)$ on a generic smooth fibre $X_t$. Thus
$$h^0(X_t,A_t)=d_1\cdots d_n$$
and the higher cohomology must vanish. We will show that this is also true for singular fibres.

\begin{claim}
Let $t$ be an arbitrary point in $\P^n$ and let $A_t$ be the restriction of the very ample line bundle $A$ to the fibre $X_t$. Then
$$h^j(X_t,A_t)=\left\{\begin{array}{ll} d_1\cdots d_n & j=0, \\ 0 & j>0. \end{array}\right.$$
\end{claim}

\begin{proof}
As already explained, $h^j(X_t,A_t)=0$ for $j>0$ and for generic $t$. Therefore the higher direct image sheaves $R^j\pi_*A$ are torsion. On the other hand, Theorem~2.2 of Hacon~\cite{hacon04} states that $R^j\pi_*(\omega_X\otimes A)$ is torsion free for $j\geq 0$.
Since $\omega_X$ is trivial, we conclude that $R^j\pi_*A$ must vanish for $j>0$.

Now Corollary~3 on page~50 of Mumford~\cite{mumford08} states that if $\H^j(X_t,A_t)$ vanishes for all $t\in\P^n$ then the natural map
$$R^{j-1}\pi_*A\otimes_{\mathcal{O}_t} k(t)\longrightarrow \H^{j-1}(X_t,A_t)$$
is an isomorphism for all $t\in\P^n$. Since $\H^{n+1}(X_t,A_t)$ vanishes for all $t\in\P^n$ for dimension reasons, and since $R^n\pi_*A$ vanishes, we conclude that $\H^n(X_t,A_t)$ vanishes for all $t\in\P^n$. Continuing by reverse induction, we find that $\H^j(X_t,A_t)$ vanishes for all $j>0$ and for all $t\in\P^n$.

Finally, $\pi:X\rightarrow\P^n$ is a flat family and
$$h^0(X_t,A_t)=\chi(X_t,A_t)$$
is topological, so for all $t\in\P^n$ we find that $h^0(X_t,A_t)$ agrees with the value $d_1\cdots d_n$ for a smooth fibre. Moreover, $\pi_*A$ is a locally free sheaf of rank $d_1\cdots d_n$ on $\P^n$.
\end{proof}

Since $A_t$ is very ample, each fibre $X_t$ is embedded in $\P(\H^0(X_t,A_t)^{\vee})$, which by the claim is isomorphic to $\P^{M-1}$, where $M=d_1\cdots d_n$. We would like to interpret this as giving a point in $\mathrm{Hilb}$, where $\mathrm{Hilb}$ denotes the component of the Hilbert scheme of subschemes of $\P^{M-1}$ which contains the abelian varieties with polarization type $(d_1,\ldots,d_n)$. But the subscheme $X_t\subset\P^{M-1}$ is only defined up to the action of $\mathrm{PGL}(M,\mathbb{C})$ on $\P^{M-1}$. To get an actual point in $\mathrm{Hilb}$ for each $t$, we need to choose a basis for $\H^0(X_t,A_t)$. To summarize, there is a principal $\mathrm{GL}(M,\mathbb{C})$-bundle $G$ on $\P^n$, which is the bundle of local frames of the rank $M$ vector bundle $\pi_*A$ on $\P^n$, and a map from $G$ to $\mathrm{Hilb}$.

To make this work for an arbitrary classifying map, we first need the rank $M$ vector bundle to extend to $\P^n$. Then we need the map from the associated $\mathrm{GL}(M,\mathbb{C})$-bundle to $\mathrm{Hilb}$ to extend.

\begin{definition}
Let $\psi\in V^{(3)}$, let $\pi:X\rightarrow\P^n\backslash\Delta_0$ be the corresponding family of abelian varieties, and let $A$ be the relative polarization for this family (restricted to each fibre $X_t$, $A_t:=A|_{X_t}$ is very ample, but $A$ need not be very ample or even ample on $X$ itself). Define $V^{(4)}\subset V^{(3)}$ to be the subset of maps such that the rank $M$ vector bundle $\pi_*A$ on $\P^n\backslash\Delta_0$ extends to a vector bundle on all of $\P^n$.
\end{definition}

\begin{lemma}
The subset $V^{(4)}$ is a constructible subset of $V^{(3)}$, i.e., it is a finite union of (Zariski) locally closed subsets in $V^{(3)}$.
\end{lemma}

\begin{proof}
If a vector bundle $W$ on $\P^n\backslash\Delta_0$ extends to $\P^n$, then the extended vector bundle will be given by $\iota_*W$, where $\iota:\P^n\backslash\Delta_0\hookrightarrow\P^n$ is the inclusion map. Therefore $V^{(4)}$ is the subset of maps for which $\iota_*W$ is locally free, where $W$ is the vector bundle $\pi_*A$ on $\P^n\backslash\Delta_0$.

Consider the universal situation. Thus we have a vector bundle $\mathcal{W}$ on $(V^{(3)}\times\P^n)\backslash\mathcal{D}_0$ which we can push forward by the inclusion
$$\iota:(V^{(3)}\times\P^n)\backslash\mathcal{D}_0\hookrightarrow V^{(3)}\times\P^n$$
to get a coherent sheaf $\iota_*\mathcal{W}$. Restricting $\iota_*\mathcal{W}$ to each $\{\psi\}\times\P^n$ gives the appropriate $\iota_*W$ corresponding to $\psi$. In other words, we obtain a family of coherent sheaves on $\P^n$ parametrized by $\psi\in V^{(3)}$.

Koll{\'a}r studied this situation, and proved that there is a decomposition of $V^{(3)}$ into finitely many locally closed subsets $V^{(3)}_j$ such that the sheaves $\iota_*W$ on $\P^n$ form a flat family as $\psi$ varies in any given $V^{(3)}_j$ (see Theorem~21 in~\cite{kollar08}). Moreover, the subset
$$V^{(4)}_j:=\{\psi\in V^{(3)}_j|\iota_*W\mbox{ is locally free}\}$$
is open in $V^{(3)}_j$. We conclude that $V^{(4)}$ is constructible, since it is the (finite) union of the locally closed subsets $V^{(4)}_j$.
\end{proof}

\begin{definition}
Let $\psi\in V^{(4)}$, let $W$ be the resulting vector bundle on $\P^n$ given by extending $\pi_*A$, and let $G$ be the principal $\mathrm{GL}(M,\mathbb{C})$-bundle on $\P^n$ given by the bundle of local frames of $W$. For $t\in\P^n\backslash\Delta_0$, each point $f\in G_t$ represents a choice of basis for $\H^0(X_t,A_t)$, so there is a morphism
$$\alpha:G|_{\P^n\backslash\Delta_0}\rightarrow\mathrm{Hilb}$$
given by mapping $f$ to the point in the Hilbert scheme $\mathrm{Hilb}$ which represents
$$X_t\hookrightarrow\P(\H^0(X_t,A_t)^{\vee})\stackrel{f}{\cong}\P^{M-1}.$$
Define $V^{(5)}\subset V^{(4)}$ to be the subset of maps $\psi$ such that $\alpha$ extends to a morphism
$$\bar{\alpha}:G\rightarrow\mathrm{Hilb}.$$
\end{definition}

\begin{lemma}
The subset $V^{(5)}$ is (Zariski) locally closed in $V^{(4)}$.
\end{lemma}

\begin{proof}
First consider a simpler situation. Suppose we have a rational map $\beta:\mathbb{C}^n\dashrightarrow\P^N$ that is regular on the complement of the origin $O$. Resolving the indeterminacy, we obtain a morphism $\hat{\beta}:\widehat{\mathbb{C}^n}\rightarrow\P^N$ where $\widehat{\mathbb{C}^n}$ is $\mathbb{C}^n$ with the origin blown up, possibly several times. Let $E\subset\widehat{\mathbb{C}^n}$ be the exceptional locus of the blow-up. When considering families of such maps $\beta$, we look at the dimension of the image of $E$ under $\hat{\beta}$. For each integer $l\geq 0$, let $U_l$ be the subset of rational maps $\beta$ such that $\hat{\beta}(E)$ has dimension greater than or equal to $l$; by semi-continuity this is a closed condition, i.e., $U_l$ is a closed subset. Therefore $U:=U_0\backslash U_1$ is a locally closed subset. But $U$ is precisely the subset of rational maps $\beta$ such that $\hat{\beta}(E)$ is a single point, which means that the rational map $\beta$ must extend to a morphism defined on all of $\mathbb{C}^n$. We have thus shown that extendability of the rational map $\beta$ is a locally closed condition.

The above argument can be adapted to show that, in general, extendability of maps is a locally closed condition. Therefore $V^{(5)}$ is locally closed in $V^{(4)}$.
\end{proof}

Suppose $\psi\in V^{(4)}$. As described above, there is an associated map
$$\alpha:G|_{\P^n\backslash\Delta_0}\rightarrow\mathrm{Hilb}.$$
The group $\mathrm{PGL}(M,\mathbb{C})$ acts on the $\mathrm{GL}(M,\mathbb{C})$-bundle $G$ fibrewise. It also acts on $\mathrm{Hilb}$, which is a component of the Hilbert scheme of subschemes of $\P^{M-1}$, by acting on the ambient space. The $\mathrm{PGL}(M,\mathbb{C})$-equivariance of $\alpha$ follows immediately from its definition.

Now if $\psi\in V^{(5)}$ then the map $\alpha$ extends to
$$\bar{\alpha}:G\rightarrow\mathrm{Hilb}.$$
We claim that $\bar{\alpha}$ will also be $\mathrm{PGL}(M,\mathbb{C})$-equivariant. To see this, fix $g\in\mathrm{PGL}(M,\mathbb{C})$. The set
$$\{f\in G|\bar{\alpha}(g.f)=g.\bar{\alpha}(f)\}$$
is closed in $G$ and it contains the dense open subset $G|_{\P^n\backslash\Delta_0}$. It follows that $\bar{\alpha}(g.f)=g.\bar{\alpha}(f)$ for all $f\in G$, and for all $g\in\mathrm{PGL}(M,\mathbb{C})$, proving the claim.

So for each point $t\in\Delta_0$, the fibre $G_t$ maps to a $\mathrm{PGL}(M,\mathbb{C})$-orbit in $\mathrm{Hilb}$, corresponding to a subscheme $X_t\subset\P^{M-1}$ defined up to projective transformations of the ambient space. Moreover, these $X_t$ fit together to give a completion $X\rightarrow\P^n$ of the original family of abelian varieties over $\P^n\backslash\Delta_0$ (we are abusing notation here, using $X$ to denote the original family and its completion, but the meaning should be clear from the context). Thus $V^{(5)}$ is the subset of maps $\psi$ which yield complete families $X\rightarrow\P^n$. Moreover, the fibres $X_t$ all have the same Hilbert polynomial, since they correspond to points in $\mathrm{Hilb}$, and the family $X\rightarrow\P^n$ is flat.

\begin{definition}
Define $V^{(6)}\subset V^{(5)}$ to be the subset of maps such that the corresponding family $X\rightarrow\P^n$ is smooth, i.e., the total space $X$ is smooth. Define $V^{(7)}\subset V^{(6)}$ to be the subset of maps such that $X$ is an irreducible holomorphic symplectic manifold, and thus $X\rightarrow\P^n$ is a Lagrangian fibration.
\end{definition}

\begin{lemma}
The subset $V^{(6)}$ is open in $V^{(5)}$ and the subset $V^{(7)}$ is open in $V^{(6)}$.
\end{lemma}

\begin{proof}
Each map $\psi\in V^{(5)}$ yields a variety $X$. Given any family of varieties, the singular ones form a closed subset; this proves that $V^{(6)}$ is open in $V^{(5)}$.

Sufficient conditions to ensure that a smooth variety $X$ of dimension $2n$ will be an irreducible holomorphic symplectic manifold are $h^{1,0}(X)=0$, $h^{2,0}(X)=1$, and $K_X$ is trivial. In a family, the first two conditions are both open (in a flat family the Hodge numbers would even be constant). Then the third condition will also be open inside the open set where $h^{1,0}(X)=0$, since this ensures that the Picard group is discrete. Therefore $V^{(7)}$ is open in $V^{(6)}$.
\end{proof}

\subsection{Conclusion of the proof}

\begin{proof}{\bf (of Theorem~\ref{main})} 
We have described an association that assigns a classifying map
$$\psi\in\bigsqcup_{k=1}^{a(n+1)}V_k$$
to each Lagrangian fibration $\pi:X\rightarrow\P^n$ satisfying hypotheses $1$ to $4$. By Proposition~\ref{injective} this association is one-to-one. By definition, the subset
$$\bigsqcup_{k=1}^{a(n+1)}V_k^{(7)}\subset\bigsqcup_{k=1}^{a(n+1)}V_k$$
describes precisely those maps $\psi$ which arise as classifying maps of Lagrangian fibrations. In other words, there is a bijection between the set of Lagrangian fibrations satisfying hypotheses $1$ to $4$ and the set $\bigsqcup_{k=1}^{a(n+1)}V_k^{(7)}$. By Lemmas 8 to 13, $\bigsqcup_{k=1}^{a(n+1)}V_k^{(7)}$ is a constructible algebraic set inside $\bigsqcup_{k=1}^{a(n+1)}V_k$, so it consists of finitely many connected components. Each connected component then parametrizes a continuous family of Lagrangian fibrations, which are therefore deformation equivalent. It follows that there are finitely many Lagrangian fibrations satisfying hypotheses $1$ to $4$, up to deformation.
\end{proof}


\section{The hypotheses}

\subsection{Rank-one semi-stable singular fibres}

Matsushita~\cite{matsushita01} classified the codimension one singular fibres which can occur in a Lagrangian fibration of dimension four. Hwang and Oguiso~\cite{ho09i, ho11} extended this classification to higher dimensions (see also Matsushita~\cite{matsushita07} for the projective case). Our hypothesis $3$, that the generic singular fibre is a rank-one semi-stable degeneration of abelian varieties, corresponds to type $I_m$ or $A_{\infty}$ in Hwang and Oguiso's notation. More precisely, their classification is of the {\em characteristic 1-cycles\/}; if the (rank-one semi-stable) singular fibre consists of $k$ irreducible components then the characteristic 1-cycle will be either of type $I_m$ where $m$ is a multiple of $k$, or of type $A_{\infty}$.

Note that a rank $r$ semi-stable degeneration of an abelian variety is a toric compactification of a semi-abelian variety with rank $r$ toric part; it occurs over a codimension $r$ boundary component of a toric compactification of $\mathcal{A}_{(d_1,\ldots,d_n)}$ (see the book of Hulek, Kahn, and Weintraub~\cite{hkw93} for the abelian surface
case). However, by Hwang and Oguiso's classification, these higher rank semi-stable degenerations cannot occur as generic singular fibres of Lagrangian fibrations. In other words, hypothesis $3$ is equivalent to requiring that the generic singular fibre be semi-stable, as it would then automatically be rank-one as well.

When $n=1$, our hypothesis is that the elliptic K3 surface has only singular fibres of type $I_m$ in Kodaira's classification. This is not true for every elliptic K3 surface, but it is true for a generic elliptic K3 surface, which in fact has only nodal rational curves (type $I_1$) as singular fibres. We expect similar behaviour in higher dimensions.

\begin{conjecture}
Let $\pi:X\rightarrow\P^n$ be a Lagrangian fibration by abelian varieties. Then a generic deformation of $X$, which preserves the Lagrangian fibration, will have rank-one semi-stable degenerations as singular fibres in codimension one. Moreover, if $\pi:X\rightarrow\P^n$ admits a global section then we can restrict to deformations that also preserve the section.
\end{conjecture}

\begin{remark}
A generic elliptic K3 surface which admits a section will have Picard number $\rho=2$. This means that, a priori, it can only have singular fibres of types $I_1$ and/or $II$, i.e., only nodal and/or cuspidal rational curves. All other singular fibres in Kodaira's classification would contribute non-trivially to the Picard group, and force $\rho>2$.

While this simple argument does not eliminate cuspidal rational curves, it does have the advantage of generalizing to higher dimensional Lagrangian fibrations. Firstly, the deformation theory is analogous to the corresponding theory for elliptic K3 surfaces (see Matsushita~\cite{matsushita05} and the author's article~\cite{sawon09}). In particular, the local Torelli theorem can be used to show that a generic Lagrangian fibration which admits a section will have Picard number $\rho=2$. Next, Oguiso~\cite{oguiso09} generalized the Shioda-Tate formula to fibrations by abelian varieties. Using this formula, one can show that most of the singular fibres in Hwang and Oguiso's classification contribute non-trivially to the Picard group. So if $\rho=2$, the Lagrangian fibration can only have generic singular fibres of types $I_m$, $A_{\infty}$, and/or $II$, i.e., only semi-stable degenerations and/or higher dimensional analogues of cuspidal rational curves. The conjecture above asserts that, generically, type $II$ singular fibres should not occur either.
\end{remark}

A more direct way to avoid the hypothesis of semi-stable singular fibres would be to generalize Theorem~\ref{main} so that it applies in the presence of any of the singular fibres on Hwang and Oguiso's (or Matsushita's) list. Now the canonical bundle formula~(\ref{Kodaira}) will apply in the general situation: the divisor $B$ precisely accounts for the singular fibres in codimension one which are not semi-stable. Moreover, Matsushita~\cite{matsushita07} has determined the contributions to $B_R$ for the various possible singular fibres. Notice that as the contribution $B_R$ of the ``bad'' singular fibres increases, the degree of the moduli part $J(X/\P^n)$ will decrease, so that we should get an even stronger bound on the degree of the classifying map $\bar{\phi}$ in these cases. The difficulty is in {\em defining\/} the extension $\bar{\phi}$ of $\phi$, since we don't have an appropriate compactification of $\mathcal{A}_{(d_1,\ldots,d_n)}$ whose boundary parametrizes all of the additional non-semi-stable singular fibres that we must allow.

\subsection{Existence of a section}

Let $\pi:X\rightarrow\P^n$ be a Lagrangian fibration that admits a global section. Away from the singular fibres, $X|_{\P^n\backslash\Delta}$ is a group scheme, as the fibres are abelian groups. Moreover, the semi-stable singular fibres are toric compactifications of semi-abelian varieties, so their smooth loci are also abelian groups. If instead $\pi:X\rightarrow\P^n$ does not admit a global section, then $X|_{\P^n\backslash\Delta}$ will be a torsor over an associated group scheme $U$ over $\P^n\backslash\Delta$. For example, $U$ could be defined as the relative Albanese variety of $X|_{\P^n\backslash\Delta}$.

\begin{question}
Can $U$ be compactified to produce a Lagrangian fibration $\pi_0:X^0\rightarrow\P^n$?
\end{question}

\begin{remark}
Assume that $X$ has semi-stable singular fibres, at least generically, i.e., over $\Delta\backslash\Delta_0$. Then we automatically obtain a partial compactification of $U$ by adding in the semi-abelian varieties over $\Delta\backslash\Delta_0$. These semi-abelian fibres can then be compactified in the same way as they are compactificed in $X$ itself, to produce a proper fibration $W$ over $\P^n\backslash\Delta_0$. There still remains the question of whether $W$ can be compactified to a Lagrangian fibration over $\P^n$. In general, there is no reason to expect such a compactification to exist (cf.\ Lemmas $12$ and $13$ above), but in this case we may be able to use the fact that the torsor $X|_{\P^n\backslash\Delta}$ can be compactified to a Lagrangian fibration $X$ over $\P^n$. Also, it is enough to show that $W$ can be compactified to an equidimensional fibration $\pi_0:X^0\rightarrow\P^n$, as the argument in Lemma~6 of~\cite{sawon08} would then show that we have a Lagrangian fibration.
\end{remark}

\begin{example}
O'Grady constructed a ten-dimensional holomorphic symplectic manifold by desingularizing a moduli space of sheaves on a K3 surface~\cite{ogrady99}. For certain K3 surfaces $S$, O'Grady's space admits a Lagrangian fibration, which was further studied by Rapagnetta~\cite{rapagnetta08}. This Lagrangian fibration is closely related to the Beauville-Mukai system on $\mathrm{Hilb}^5S$. Indeed, over $\P^5\backslash\Delta$, the Beauville-Mukai system is a torsor over the O'Grady-Rapagnetta Lagrangian fibration (see Section~5.3 of~\cite{sawon04}). These two fibrations are also locally isomorphic as fibrations over $\P^5\backslash\Delta_0$. However, it is worth noting that they have non-isomorphic singular fibres in codimension two; i.e., we add different singular fibres over $\Delta_0$ in order to compactify these two fibrations over $\P^5\backslash\Delta_0$.
\end{example}

Let us now assume that the above question has been answered in the affirmative. Thus to each Lagrangian fibration $\pi:X\rightarrow\P^n$ we can associate a Lagrangian fibration $\pi_0:X^0\rightarrow\P^n$ such that $X|_{\P^n\backslash\Delta}$ is a torsor over $X^0|_{\P^n\backslash\Delta}$. We call $X$ a {\em compactified torsor} over $X^0$. Note that $\pi_0:X^0\rightarrow\P^n$ admits a section, or at least a rational section (over $\P^n\backslash\Delta$).

\begin{conjecture}
Up to deformation, there are finitely many compactified torsors $\pi:X\rightarrow\P^n$ associated to a fixed Lagrangian fibration $\pi_0:X^0\rightarrow\P^n$.
\end{conjecture}

\begin{example}
Let $S\rightarrow\P^1$ be an elliptic K3 surface with only nodal rational curves as singular fibres. Then $S$ is a compactified torsor over its relative Jacobian $S^0\rightarrow\P^1$. The set of all compactified torsors over a fixed $S^0$ forms a group known as the {\em Tate-Shafarevich group\/}. For K3 surfaces it is isomorphic to the analytic Brauer group $\H^2(S^0,\mathcal{O}^*)$, which fits into the exact sequence
$$\ldots\rightarrow\H^2(S^0,\mathcal{O})\rightarrow\H^2(S^0,\mathcal{O}^*)\rightarrow\H^3(S^0,\mathbb{Z})\rightarrow\ldots$$
Since $\H^2(S^0,\mathcal{O})\cong\mathbb{C}$ and $\H^3(S^0,\mathbb{Z})$ vanishes, the Brauer group is connected, and hence all compactified torsors over $S^0$ are deformation equivalent (of course, all K3 surfaces are deformation equivalent).
\end{example}

\begin{example}
An elliptic Calabi-Yau threefold $Y$ will also have a relative Jacobian $Y^0$, but in this case we find that $\H^2(Y^0,\mathcal{O})$ vanishes while $\H^3(Y^0,\mathbb{Z})$, and hence $\H^2(Y^0,\mathcal{O}^*)$, is discrete. Gross~\cite{gross94} shows that the Tate-Shafarevich group is finite in this case.
\end{example}

\begin{example}
For Lagrangian fibrations, $\H^2(X^0,\mathcal{O})\cong\mathbb{C}$ while $\H^3(X^0,\mathbb{Z})$ may or may not vanish. In some cases, the Tate-Shafarevich group is connected (see~\cite{sawon08ii}). In other cases, it has several connected components (the Beauville-Mukai system in ten dimensions is a compactified torsor over the O'Grady-Rapagnetta fibration, but they are {\em not} deformation equivalent). In general, we expect the number of compactified torsors over $X^0$ (up to deformation) to be bounded by a formula involving the number of irreducible components of the generic singular fibre of $\pi_0:X^0\rightarrow\P^n$. We assume that the generic singular fibre is semi-stable, so this is equal to the number of connected components of the corresponding semi-abelian variety; we believe this can be related to the number of components of the Brauer group.
\end{example}

\subsection{Maximal variation of abelian varieties}

Our fourth hypothesis was that $\pi:X\rightarrow\P^n$ describes a maximal variation of abelian varieties around a generic point $t\in\P^n$. In other words, a sufficiently small neighbourhood $V$ of $t$ is embedded in $\mathcal{A}_{d_1,\ldots,d_n}$ by the classifying map $\phi$, or equivalently, the derivative $(d\phi)_t$ has rank $n$ at a generic point $t\in\P^n$.

\begin{example}
If an elliptic K3 surface does {\em not\/} describe a maximal variation then $(d\phi)_t$ must have rank $0$ at a generic point, but this means that $d\phi$ vanishes identically and $\phi$ is constant. An example of such an elliptic K3 surface, with constant functional invariant, is given by constructing a Kummer K3 surface from an abelian surface $A$ which is fibred over an elliptic curve $E$. This gives an elliptic fibration
$$\mathrm{Blow}(A/\pm 1)\rightarrow A/\pm 1 \rightarrow E/\pm 1\cong\P^1$$
which is locally isotrivial: every smooth fibre is isomorphic to a fixed elliptic curve.

Note that this elliptic K3 surface has four singular fibres of Kodaira type $I^*_0$, sitting above the four branch points of the double cover $E\rightarrow\P^1$. These are certainly {\em not\/} semi-stable singular fibres. Indeed, a locally isotrivial elliptic K3 surface cannot have semi-stable singular fibres: otherwise the functional invariant $J$ would extend to all of $\P^1$, and since it is constant, every fibre would be smooth. So an elliptic K3 surface with semi-stable singular fibres must also describe a maximal variation.
\end{example}

We expect a similar dichotomy in higher dimensions.

\begin{conjecture}[Matsushita]
Let $\pi:X\rightarrow\P^n$ be a Lagrangian fibration. After choosing a local section over a neighbourhood of a generic point $t\in\P^n$, we can define a classifying map $\phi$ locally. Then we either have a maximal variation (the image of $\phi$ in $\mathcal{A}_{d_1,\ldots,d_n}$ is $n$-dimensional; equivalently, $(d\phi)_t$ is of rank $n$) or a minimal variation ($\phi$ is constant and $d\phi$ vanishes).
\end{conjecture}

\begin{remark}
As with elliptic K3 surfaces, a Lagrangian fibration which is both a minimal variation and which has semi-stable singular fibres must be smooth up to codimension one, which is impossible. So Matushita's conjecture implies that a Lagrangian fibration whose generic singular fibres are semi-stable must also describe a maximal variation.
\end{remark}

\begin{flushleft}
Department of Mathematics\hfill sawon@email.unc.edu\\
University of North Carolina\hfill www.unc.edu/$\sim$sawon\\
Chapel Hill NC 27599-3250\\
USA\\
\end{flushleft}

\end{document}